\documentclass[11pt,twoside, leqno]{article}

\usepackage{amsthm}
\usepackage{amssymb}
\usepackage{amsmath}
\usepackage{mathrsfs}
\usepackage{txfonts}
\usepackage{graphics}

\allowdisplaybreaks 
\def\rr{{\mathbb R}}

\def\fz{\infty}
\def\az{\alpha}
\def\supp{{\mathop\mathrm{\,supp\,}}}
\def\dist{{\mathop\mathrm {\,dist\,}}}
\def\loc{{\mathop\mathrm{\,loc\,}}}
\def\Lip{{\mathop\mathrm{\,Lip\,}}}

\def\lz{\lambda}

\def\ez{\epsilon}

\def\kz{\kappa}

\def\gz{{\gamma}}

\def\sz{\sigma}

\def\hs{\hspace{0.3cm}}

\def\r{\right}
\def\lf{\left}

\newtheorem{thm}{Theorem}[section]
\newtheorem{lem}{Lemma}[section]
\newtheorem{prop}{Proposition}[section]
\newtheorem{rem}{Remark}[section]

\numberwithin{equation}{section}

\begin{document}
\arraycolsep=1pt
\author{Renjin Jiang} \arraycolsep=1pt
\title{{\Large\bf Cheeger-harmonic functions in metric measure \\ spaces
revisited }
 \footnotetext{
\endgraf}}

\date{ }
\maketitle
\begin{center}
\begin{minipage}{11cm}\small
{\noindent{\bf Abstract}. Let $(X,d,\mu)$ be a complete metric
measure space, with $\mu$ a locally doubling measure, that supports
a local weak $L^2$-Poincar\'e inequality. By assuming a heat
semigroup type curvature condition, we prove that Cheeger-harmonic
functions are  Lipschitz continuous on $(X,d,\mu)$. Gradient
estimates for Cheeger-harmonic functions and solutions to
a class of non-linear Poisson type equations
are presented. }\end{minipage}
\end{center}
\vspace{0.2cm}

\section{Introduction}
\hskip\parindent
The study of Lipschitz continuity
of Cheeger-harmonic functions was originated by Koskela et. al. \cite{krs}, which can be viewed
as a metric version of Yau's gradient estimate (\cite{ya75,chy}).
In \cite{krs} it is proved that on an Ahlfors regular metric spaces, an $L^2$-Poincar\'e inequality
and a heat semigroup type curvature condition are sufficient to guarantee Lipschitz
continuity of Cheeger-harmonic functions. Later, a quantitative gradient estimate
for Cheeger-harmonic functions was given in \cite{ji2}. The main aim of this paper
is to relax the Ahlfors regularity in \cite{krs,ji2} to doubling of the measure.
Besides this, gradient estimates for a class of non-linear Poisson type equations
are also given.

Let $(X,d,\mu)$ be a complete,
pathwise connected metric space, equipped with a locally doubling
measure $\mu$, i.e., for each $R_0>0$, there exists a constant
$C_d(R_0)$ such that for each $0<r<R_0/2$ and all $x$,
\begin{equation}\label{1.1}
\mu(B(x,{2r}))\le C_d(R_0)\mu(B(x,r)).
\end{equation}
We then call the measure locally $Q$-doubling for some $Q>0$, if
for each $R_0>0$, there exists a constant $C_Q(R_0)$ such that
such that for every $x\in X$ and all $0<r<R\le R_0$, it holds
\begin{equation}\label{1.2}
\mu(B(x,{R}))\le C_Q(R_0)\lf(\frac{R}{r}\r)^Q\mu(B(x,{r})).
\end{equation}
We say that  $\mu$ is globally $Q$-doubling if the above holds with a
constant that is independent of $R_0.$

Throughout the paper, we additionally require that $(X,d,\mu)$
is stochastically complete (see Section 2 below). The requirement comes from the technique used in the proof, and does not
look like a very natural condition; on the other hand, it is satisfied
on metric spaces with (Lott-Sturm-Villani) finite dimensional Ricci curvature
bounded from below.

An $L^2$-Poincar\'e inequality is  needed. Precisely, we assume
that $(X,d,\mu)$ supports a local weak $L^2$-Poincar\'e inequality,
i.e., for each $R_0>0$, there exists $C_P(R_0)>0$ such
that for all Lipschitz functions $u$ and each ball $B(x,{r})=B(x,r)$
with $r<R_0$,
\begin{equation}\label{1.3}
\fint_{B(x,{r})}|u-u_B|\,d\mu\le C_P(R_0)r\lf(\fint_{B(x,{2r})}[\Lip
u]^2\,d\mu\r)^{1/2},
\end{equation}
where and in what follows, $u_B=\fint_B
u\,d\mu=\mu(B)^{-1}\int_{B}u\,d\mu$, and
\begin{equation*}
\Lip u(x)=\limsup_{r\to 0}\sup_{d(x,y)\le r}\frac{|u(x)-u(y)|}{r}.
\end{equation*}
We then say that $(X,d,\mu)$ supports a uniform weak $L^2$-Poincar\'e inequality,
if \eqref{1.3} holds with a uniform constant $C_P$ for all $R_0>0$.
According to \cite{kei03} the Poincar\'e inequality here coincides with the one for all
measurable functions and their upper gradients, as introduced in \cite{hek}.

For a domain $\Omega\subset X$, the Sobolev space $H^{1,2}(\Omega)$
is defined to be the completion of all locally Lipschitz continuous functions $u$
on $\Omega$ under the norm
$$\|u\|_{H^{1,2}(\Omega)}:=\|u\|_{L^2(\Omega)}+\|\Lip u\|_{L^2(\Omega)}.$$
An important fact to us is that for each $u\in H^{1,2}(\Omega)$ we
can assign a (Cheeger) derivative $Du$ by \cite{ch}. This derivative
operator is linear, satisfies the Leibniz rule, and there is an
inner product norm that is comparable to our original norm: for a
locally Lipschitz function $u,$ $Du\cdot Du$ is uniformly almost
everywhere comparable to to the square of the local Lipschitz constant
$\Lip u,$ see Section 2 below. Notice that in many concrete settings,
one can make a different choice of an operator that satisfies the above conditions.
We call any operator $D$
that has the above properties a Cheeger derivative operator.

We next define the {\it
Cheeger-Laplace} equation on $(X,d,\mu)$.
For a domain $\Omega$, we say that $u\in H^{1,2}(\Omega)$ is a solution to
the equation $L u=g(x, u)$ in
$\Omega$, if
\begin{equation*}
-\int_\Omega Du(x)\cdot D\phi(x)\,d\mu(x)=\int_\Omega
g(x,u)\phi(x)\,d\mu(x)
\end{equation*}
holds for all  Lipschitz functions $\phi$ with compact support in $\Omega$,
where $g(x,t)$ is a measurable function defined on $X\times \rr$ and continuous with respect to the variable
$t$. If  $L u=0$ in $\Omega$, then we say that $u$ is {\it Cheeger-harmonic}
in $\Omega$.

The above Dirichlet problem and related parabolic equations
have been widely studied; see \cite{bm,bb10,kms,sal,st1,st2,st3}
for instance. According to \cite{sal,st3},
the doubling condition and validity of an $L^2$-Poincar\'e inequality are equivalent
to a parabolic Harnack inequality, which further
implies an elliptic Harnack inequality and hence the H\"older continuity
of harmonic functions (see \cite{bm,st3}).

However, Lipschitz regularity does not follow from doubling and Poincar\'e
inequality, see the examples from the introduction of \cite{krs}.
Thus, some additional requirement is needed for Lipschitz regularity
of solutions.

For $f$ and $g$ in $H^{1,2}(X)$, define the bilinear form
$\mathscr{E}$ by
$$\mathscr{E}(f,g)=\int_X Df(x)\cdot Dg(x)\,d\mu(x).$$
Corresponding to such a form there exists an infinitesimal generator
$A$ which acts on a dense subspace $D(A)$ of $H^{1,2}(X)$, and there
is a semigroup $\{T_t\}_{t>0}$ generated by $A$; see Section 2 below.

We say that $(X,D,\mu)$ satisfies {\it heat semigroup curvature
condition} for our Cheeger derivative operator $D$, if there exists a
nonnegative and nondecreasing function $c_{\kz}(T)$ on $(0,\fz)$ such that for each
$0<t<T$ and every $g\in H^{1,2}(X)$, it holds
\begin{eqnarray}\label{1.4}
T_t (g^2)(x)-[T_t(g)(x)]^2&&\le (2t+c_{\kz}(T)t^2)T_t(|Dg|^2).
\end{eqnarray}

Let us state the first gradient estimate.
\begin{thm}\label{t1.1}
Let $(X,d)$ be a stochastically complete metric space with a locally $Q$-doubling measure
$\mu$, $Q\in (1,\fz)$. Assume that $(X,d,\mu)$ supports a local
weak $L^2$-Poincar\'e inequality and the heat semigroup curvature condition
\eqref{1.4}.

Let $u$ be a solution
to the equation $Lu=-\lz u$ in $2B$, where $B=B(y_0,R)$ and $\lz\in L^\fz(2B)$.
Then there exists $C=C(Q,C_Q(2R),C_P(2R),\|\lz\|_{L^\fz(2B)} R^2)$ such that
$$\||Du|\|_{L^\fz{(B)}}\le C\lf(\frac 1{R}+\sqrt
{c_{\kz}(R^2)}\r)\fint_{2B}|u|\,d\mu.$$
\end{thm}

The above estimate in particular implies that Cheeger-harmonic
functions are locally Lipschitz continuous under the above assumptions.

Let us revisit an example from \cite{krs}.
Consider the metric space $(\Omega,d)$ with
$\Omega=[-1,1]\times [-1,1]\subset \rr^2$ and $d$ the Euclidean metric.
Let $w(x,y)=\sqrt{|x|}$.
Set $\,d\mu=w\,dx$. Then $(\Omega,d,\mu)$ supports an $L^2$-Poincar\'e inequality and $\mu$ is a
doubling measure. The function $u(x,y)=\mathrm{sgn\,}(x)\sqrt{|x|}$ is harmonic in $\Omega$, but it is
not locally Lipschitz in $\Omega$. It was understood in \cite{krs} that  in order
to deduce Lipschitz regularity, the doubling condition should be strengthened
to Ahlfors regularity.
According to Theorem \ref{t1.1}, the reason that the Lipschitz regularity
of Cheeger-harmonic functions fails is due to lack of lower curvature bounds
rather than to lack of Ahlfors regularity.

We have the following gradient estimates for positive Cheeger-harmonic functions.

\begin{thm}\label{t1.2}
Let $(X,d)$ be a stochastically complete metric space with a locally $Q$-doubling measure
$\mu$, $Q\in (1,\fz)$. Assume that $(X,d,\mu)$ supports a local
weak $L^2$-Poincar\'e inequality and the heat semigroup curvature condition
\eqref{1.4}.
Let $u$ be a positive Cheeger-harmonic function in $2B$, where $B=B(y_0,R)$.

 (i) There exists $C=C(Q,C_Q(2R),C_P(2R))>0$ such that for almost every $x\in B$, it holds
  \begin{equation*}
\frac{|Du(x)|}{u(x)}\le C\lf(\sqrt {c_\kz(R^2)}+\frac{1}{R}\r);
  \end{equation*}

  (ii) If $c_\kz(1)>0$, then there exists a fixed constant $C=C(Q,C_Q(1),C_P(1))$
  such that for almost every $x\in B$, it holds
  \begin{equation*}
\frac{|Du(x)|}{u(x)}\le C\lf(\sqrt{c_\kz(1)}+\frac{1}{R}\r).
  \end{equation*}
\end{thm}

Examples that satisfy assumptions in the above theorems were discussed in
\cite{krs,ji2}. Here we point out that, as a consequence of relaxing
the Ahlfors regularity from \cite{krs,ji2}, the assumptions are satisfied on
finite dimensional Riemannian manifolds with Ricci curvature bounded from below,
weighted Riemannian manifold with Bakry-Emery's curvature bounded from below,
as well as compact Alexandrov
spaces with curvature bounded from below;
see \cite{bak1,be2,chy,gro1,gko11}.

Notice that Zhang and Zhu \cite{zz11} have proved Yau's gradient estimate on
Alexandrov spaces with a new Ricci curvature condition (see \cite{zz10}).

On a complete metric space satisfying Lott-Sturm-Villani's curvature condition
$CD(K,N)$ for some $K\in \rr$ and $N\in (1,\fz)$ (see \cite{lv09,stm5},
in \cite{lv09} only $CD(0,N)$ condition is introduced), Sturm \cite[Corollary 2.4]{stm5} shows that a
local doubling condition holds, and  a global doubling condition holds
if $K\ge 0$. Moreover, it is proved by Rajala \cite{raj1,raj2} that
a local weak $L^2$-Poincar\'e inequality holds on them, and a uniform $L^2$-Poincar\'e
inequality holds if $K\ge 0$.

However, as $CD(K,N)$ conditions include the Finsler geometry,
it is not known if the heat semigroup curvature condition holds under them.  Recently,
Ambrosio et. al. \cite{ags1} (see also \cite{agmr}) introduced a Riemannian
Ricci curvature condition $RCD(K,\fz)$ on metric spaces, under which Bakry-Emery's
curvature condition holds (see \cite[Theorem 6.2]{ags1}) for the {\it minimal
weak upper gradient}.
The heat semigroup curvature condition then holds under $RCD(K,\fz)$ conditions
via an argument of Bakry \cite{bak1}.

Consequently, the gradient estimates in Theorems \ref{t1.1} and \ref{t1.2} apply
on metric spaces satisfying both $CD(K,N)$ and $RCD(K,\fz)$.

The paper is organized as follows. In Section 2, we give some basic notation
and notions for Cheeger derivatives, Dirichlet forms and
heat kernels. Section 3 is devoted to establishing gradient estimates for
equations of type $Lu=g(x,u)$ with the assumption that $g(x,u)$ is bounded.
The main results are proved in Section 4.

Finally, we make some conventions. Throughout the paper, we denote
by $C,c$ positive constants which are independent of the main
parameters, but which may vary from line to line. The symbol
$B(x,R)$ denotes an open ball
with center $x$ and radius $R$, and $CB(x,R)=B(x,CR).$

\section{Preliminaries}

\subsection{Cheeger Derivative in metric measure spaces}
\hskip\parindent The following result due to Cheeger \cite{ch}
gives us a derivative operator on metric measure spaces.

\begin{thm}\label{t2.1}
Assume that $(X,\mu)$ supports a local weak $L^2$-Poincar\'e inequality
and that $\mu$ is doubling. Then there exists $N > 0$,
depending only on the doubling constant and the constants in the
Poincar\'e inequality, such that the following holds. There exists a
countable collection of measurable sets $U_\az$, $\mu(U_\az)>0$ for
all $\az$, and Lipschitz functions $X^\az_1,\cdots,
X^\az_{k(\az)}:\, U_\az\to \rr$, with $1\le k(\az)\le N$ such that
$\mu\lf(X \setminus\cup_{\az=1}^\fz U_\az\r)=0$, and for all $\az$
the following holds: for $f: X\to \rr$ Lipschitz, there exist
$V_\az(f)\subseteq U_\az$ such that $\mu(U_\az\setminus
V_\az(f))=0$, and Borel functions
$b_1^\az(x,f),\cdots,b^\az_{k(\az)}(x,f)$ of class $L^\fz$ such that
if $x\in V_\az(f)$, then
$$\Lip(f-a_1 X_1^\az-\cdots-a_{k(\az)}X^\az_{k(\az)})(x)=0$$
if and only if
$(a_1,\cdots,a_{k(\az)})=(b^\az_1(x,f),\cdots,b^\az_{k(\az)}(x,f))$.
Moreover, for almost every $x\in U_{\az_1}\cap U_{\az_2}$, the
``coordinate functions" $X_i^{\az_2}$ are linear combinations of the
$X_i^{\az_1}$'s.
\end{thm}
By Theorem \ref{t2.1}, for each Lipschitz function $u$ we can assign
a Cheeger derivative $Du$, and for each
locally Lipschitz $f$, $\Lip f$ is are comparable to $|Du|$
almost everywhere.

By \cite{sh} and \cite{ch}, the Sobolev spaces $H^{1,2}(X)$ are
isometrically equivalent to the Newtonian Sobolev spaces
$N^{1,2}(X)$ defined in \cite{sh}. For a domain $\Omega\subset X$, following \cite{kkm}, we define the
Sobolev space
with zero boundary values $H^{1,2}_0(\Omega)$ to be the space of those
$u\in H^{1,2}(X)$
for which $u\chi_{X\setminus \Omega}$ vanishes except a set of $2$-capacity zero.
 A useful fact is
that the Cheeger derivative satisfies the Leibniz rule, i.e., for
all $u,v\in H^{1,2}(X)$,
$$D(uv)(x)=u(x)Dv(x)+v(x)Du(x).$$

\subsection{Dirichlet forms and heat kernels}
\hskip\parindent In this subsection, we recall the Dirichlet
forms and heat kernels from \cite{bm,st1,st2,st3}.
Define the bilinear form $\mathscr{E}$ by
$$\mathscr{E}(f,g)=\int_X Df(x)\cdot Dg(x)\,d\mu(x)$$
with the domain $D(\mathscr{E})=H^{1,2}(X)$. It is easy to see that
$\mathscr{E}$ is symmetric and closed. Corresponding to such a form
there exists an infinitesimal generator $A$ which acts on a dense
subspace $D(A)$ of $H^{1,2}(X)$ so that for all $f\in D(A)$ and each
$g\in H^{1,2}(X)$,
$$\int_X g(x)Af(x)\,d\mu(x)=-\mathscr{E}(g,f).$$

From \cite{krs}, we have the following Leibniz rule for Dirichlet
forms.
\begin{lem}
If $u,\,v\in H^{1,2}(X)$, and $\phi\in H^{1,2}(X)$ is a bounded
Lipschitz function, then
\begin{equation*}
\mathscr{E}(\phi,uv)=\mathscr{E}(\phi u,v)+\mathscr{E}(\phi
v,u)-2\int_X \phi Du(x)\cdot Dv(x)\,d\mu(x).
\end{equation*}
Moreover, if $u,\,v\in D(A)$, then we can unambiguously define the
measure $A(uv)$ by setting
\begin{equation*}
A(uv)=uAv+vAu+2Du\cdot Dv.
\end{equation*}
\end{lem}

%
%
%
%
%
%
Associated with the Dirichlet form $\mathscr{E}$, there is a
semigroup $\{T_t\}_{t>0}$, acting on $L^2(X)$.
Moreover, there is a heat kernel $p$ on $X$, which is a measurable
function  on $\rr\times X\times X$ and satisfies
$$ T_tf(x)=\int_X f(y)p(t,x,y)\,d\mu(y)$$ for
every $f\in L^2(X,\mu)$ and all $t\ge 0$, and $p(t,x,y)=0$ for every
$t<0$. Under the facts that the measure on $X$ is locally doubling
and a local $L^2$-Poincar\'e inequality holds, Sturm (\cite{st2,st3})
proved a Gaussian estimate for
the heat kernel, i.e., for each $t<R^2$ and all $x,y\in X$,
\begin{eqnarray}\label{2.1}
  &&C^{-1}\mu(B(x,{\sqrt t}))^{-1/2}\mu(B(y,{\sqrt t}))^{-1/2}\exp\lf\{-\frac{d(x,y)^2}{C_2t}\r\}\nonumber\\
  &&\hs \hs\hs\le p(t,x,y)\le
  C\mu(B(x,{\sqrt t}))^{-1/2}\mu(B(y,{\sqrt t}))^{-1/2}\exp\lf\{-\frac{d(x,y)^2}{C_1t}\r\},
\end{eqnarray}
where is $C$ depends on $C_Q(R)$ and $C_P(R)$. Notice that when
$x,y$ are sufficient close, i.e., $d(x,y)<R$, then \eqref{2.1} can
be written as
\begin{eqnarray}\label{2.2}
  &&C^{-1}\mu(B(x,{\sqrt t}))^{-1}\exp\lf\{-\frac{d(x,y)^2}{C_2't}\r\}\le p(t,x,y)\le
  C\mu(B(x,{\sqrt t}))^{-1}\exp\lf\{-\frac{d(x,y)^2}{C_1't}\r\}.
\end{eqnarray}
Moreover, if the measure is globally doubling and a uniform
$L^2$-Poincar\'e inequality holds, the estimates
\eqref{2.1} and \eqref{2.2} then hold for every $t>0$ and all $x,y\in X$.

By the assumption that the metric space is stochastically complete,
we know that the heat kernel is a probability measure, i.e.,
for each $x\in X$ and $t>0$,
\begin{equation}\label{2.3}
T_t1(x)=\int_X p(t,x,y)\,d\mu(y)=1.
\end{equation}
Notice that heat kernel is a probability measure if the measure
$\mu$ on a ball $B(x,{r})$ does not grow faster than $e^{cr^2}$ (see \cite{st1}).

The following lemma is essentially a Caccioppoli type inequality for
heat equations.
\begin{lem}\label{l2.2}
There exist $c, C>0$ such that for every $x\in X$,
$$\int_0^s \int_{B(x,{2R})\setminus B(x,{R})}|D_yp(t,x,y)|^2
\,d\mu(y)\,dt\le C\mu(B(x,R))^{-1}e^{-{cR^2}/{s}},$$ whenever $R>0$ and
$s\in (0,R^2]$.
\end{lem}
\begin{proof}
Let $x\in X$ be fixed and $$\phi_x(y):=\min\lf\{1, \frac 1R\dist(y,
X\setminus B(x,{3R})), \frac 2R\dist(y, B(x,{R/2}))\r\}$$ for every
$y\in X$. Then $|D\phi_x(y)|\le C/R$ and we have
\begin{eqnarray*}
&&\int_0^s \int_{X}D_yp(t,x,y)\cdot D_y(p(t,x,y)\phi_x^2(y))\,d\mu(y)\,dt\\
&&= \int_0^s \int_{X}\lf\{|D_yp(t,x,y)|^2\phi_x^2(y)-
|D_yp(t,x,y)|p(t,x,y)\phi_x(y)|D\phi_x(y))|\r\}\,d\mu(y)\,dt\\
&&\ge \int_0^s \int_{X}\lf\{\frac 12|D_yp(t,x,y)|^2\phi_x^2(y)-4p(t,x,y)^2|D\phi_x(y)|^2\r\}\,d\mu(y)\,dt\\
&&\ge \frac 12\int_0^s \int_{B(x,{2R})\setminus
B(x,{R})}|D_yp(t,x,y)|^2 \,d\mu(y)\,dt-\frac C{R^2}\int_0^s
\int_{B(x,{3R})\setminus B(x,{R/2})}p(t,x,y)^2\,d\mu(y)\,dt.
\end{eqnarray*}
As for every $y\in \supp \phi_x$, $d(x,y)<3R$. By using
the doubling condition, \eqref{2.2} and \eqref{2.3}, we further deduce that
\begin{eqnarray*}
&&\frac C{R^2}\int_0^s \int_{B(x,{3R})\setminus B(x,{R/2})}p(t,x,y)^2\,d\mu(y)\,dt\\
&&\le \frac C{R^2} \int_0^s \int_{B(x,{3R})\setminus B(x,{R/2})}
\mu(B(x,\sqrt t))^{-1}\exp\lf\{-\frac{d(x,y)^2}{C_1t}\r\}p(t,x,y)\,d\mu(y)\,dt\\
&&\le \frac C{R^2} \int_0^s \mu(B(x,R))^{-1}\frac{R^Q}{t^{Q/2}}
\exp\lf\{-\frac{R^2}{ct}\r\}\lf(\int_{X}p(t,x,y)\,d\mu(y)\r)\,dt\\
&&\le \frac C{R^2} \mu(B(x,
R))^{-1}\exp\lf\{-\frac{R^2}{cs}\r\}\int_0^s\,dt\le  C\mu(B(x,
R))^{-1}\exp\lf\{-\frac{R^2}{cs}\r\}.
\end{eqnarray*}
On the other hand, notice that $\phi_x^2(y)=0$ on $B(x,{R/2})$. By using the property
of heat semigroup, we have
\begin{eqnarray*}
&&\int_0^s \int_{X}D_yp(t,x,y)\cdot D_y(p(t,x,y)\phi_x^2(y))\,d\mu(y)\,dt\\
&&=-\int_0^s \int_{X}\frac{\partial}{\partial
t}p(t,x,y)p(t,x,y)\phi_x^2(y)\,d\mu(y)\,dt =-\frac
12\int_{X}p(s,x,y)^2\phi_x^2(y)\,d\mu(y)\le 0
\end{eqnarray*}
Combining the above estimates, we see that
\begin{eqnarray*}
&&\int_0^s \int_{B(x,{2R})\setminus B(x,{R})}|D_yp(t,x,y)|^2
\,d\mu(y)\,dt \\
&&\le \frac C{R^2}\int_0^s \int_{B(x,{3R})\setminus
B(x,{R/2})}p(t,x,y)^2\,d\mu(y)\,dt
+\int_0^s \int_{X}D_yp(t,x,y)\cdot D_y(p(t,x,y)\phi_x^2(y))\,d\mu(y)\,dt\\
&&\le C\mu(B(x,R))^{-1}\exp\lf\{-\frac{R^2}{cs}\r\},
\end{eqnarray*}
which proves the lemma.
\end{proof}

\section{From H\"older to Lipschitz}
\hskip\parindent
The main aim of this section is to prove the following result.

\begin{thm}\label{t3.1}
Let $(X,d)$ be a stochastically complete metric space with a locally $Q$-doubling measure
$\mu$, $Q\in (1,\fz)$. Assume that $(X,d,\mu)$ supports a local
weak $L^2$-Poincar\'e inequality and the heat semigroup curvature condition
\eqref{1.4}.

Let $u$ be a solution to the equation
$L u=g(x, u)$ in $\Omega\subset X$ with $g(x,u)\in L^\fz_{\loc}(\Omega)$.
Then $u$ is locally Lipschitz in $\Omega$. More precisely, for each ball
$B=B(y_0,{R})$ with $8B\subset\subset \Omega$ there
exists $C=C(Q,C_Q(8R),C_P(8R))$ such that
  \begin{equation*}
\|Du\|_{L^\fz(B)}\le
C\lf(\frac 1{R}+\sqrt{c_{\kz}(R^2)}\r)\lf[\|u\|_{L^\fz(8B)}+ R^2\|g(\cdot,u)\|_{L^\fz(8B)}\r].
  \end{equation*}
  \end{thm}

%

%
%
%
%

Throughout this section, we will always let the assumptions and notions
be the same as in Theorem \ref{t3.1} unless otherwisely stated. Moreover,
we always let $\psi$ be a cut-off function, which is Lipschitz and $\psi=1$ on
$B(y_0,2R)$, $\supp \psi\subset B(y_0,4R)$ and $|D\psi|\le\frac
{C}R$.

The following functional is the main tool for us; see
\cite{ck,krs,ji2}. Let $x_0\in B=B(y_0,R)$. For all $t\in (0,R^2)$,
define
\begin{eqnarray}\label{3.1}
J(x_0,t):=&&\frac{1}{t}\int_0^t\int_X |D(u\psi)(x)|^2
p(s,x_0,x)\,d\mu(x)\,ds.
\end{eqnarray}

%

\begin{lem}\label{l3.1}
The solution $u$ is locally
H\"older continuous in $\Omega$. More precisely, there exists $\gz\in(0,1)$
such that for almost all $x,y\in 2B=B(y_0,{2R})$, it holds
$$|u(x)-u(y)|\le C\lf(\|u\|_{L^\fz(4B)}+R^2\|g(\cdot,u)\|_{L^\fz(4B)}\r)\frac{d(x,y)^\gz}{R^\gz},$$
where $C=C(Q,C_Q(4R),C_P(4R))$.
\end{lem}
\begin{proof}
As $g(x,u)\in L^\fz(4B)$, from \cite{bm}, there exists $v\in H^{1,2}_0(4B)$
such that $Lv=g(x,u)\in 4B$. By \cite{bm}, we see that
$$\|v\|_{L^{\fz}(4B)}\le CR^2\|g(\cdot,u)\|_{L^\fz(4B)},$$
and there is $\gz_1\in (0,1)$, independent of $u,g,B$, such that for almost all $x,y\in 2B$,
$$|v(x)-v(y)|\le CR^2\|g(\cdot,u)\|_{L^\fz(4B)}\lf(\frac{d(x,y)}{R}\r)^{\gz_1}.$$

Moreover, since $u-v$ is harmonic in $4B$, we deduce from \cite[corollary 1.2]{bm} that
\begin{eqnarray*}
|(u-v)(x)-(u-v)(y)|&&\le C\lf(\fint_{4B}|u-v|\,d\mu\r)\lf(\frac{d(x,y)}{R}\r)^{\gz_2}\\
&&\le C\lf(\|u\|_{L^\fz(4B)}+R^2\|g(\cdot,u)\|_{L^\fz(4B)}\r)\lf(\frac{d(x,y)}{R}\r)^{\gz_2},
\end{eqnarray*}
for some $\gz_2\in (0,1)$. By letting $\gz=\min\{\gz_1,\gz_2\}$, we complete the proof.
\end{proof}

\begin{lem}\label{l3.2}
%
There exists $C=C(Q,C_Q(4R),C_P(4R))>0$ such that
for almost all $x_0\in B$, $x\in 2B$ and all $t\in (0,R^2)$, it holds
$$|u\psi(x)-T_t(u\psi)(x_0)|\le C\lf(\|u\|_{L^\fz(4B)}+R^2\|g(\cdot,u)\|_{L^\fz(4B)}\r)
R^{-\gz}(d(x,x_0)^\gz+t^{\gz/2}),$$ and
$$T_t(|u\psi(\cdot)-(u\psi)(x_0)|)(x_0)\le C\lf(\|u\|_{L^\fz(4B)}+R^2\|g(\cdot,u)\|_{L^\fz(4B)}\r)
R^{-\gz}t^{\gz/2},$$
\end{lem}

\begin{proof}
Let $C(u,g)=\|u\|_{L^\fz(4B)}+R^2\|g(\cdot,u)\|_{L^\fz(4B)}$.
From the previous lemma, we see that for almost all $x_0,x\in 2B$,
$$|u(x)-u(x_0)|\le CC(u,g)
\lf(\frac{d(x,x_0)}{R}\r)^\gz.$$
Thus for almost all $x_0\in B$, $x\in 2B$ and all $t\in (0,R^2)$, by
using the doubling condition and the estimate \eqref{2.1} we obtain
\begin{eqnarray*}
T_t(|(u\psi)(x_0)-u\psi(\cdot)|)(x_0)&&\le \int_{2B} |u(x_0)\psi(x_0)-u(y)\psi(y)|p(t,x_0,y)\,d\mu(y)\nonumber\\
&&\hs+\int_{X\setminus 2B} |u(x_0)\psi(x_0)-u(y)\psi(y)|p(t,x_0,y)\,d\mu(y)\nonumber\\
&&\le CC(u,g)\int_{2B} \frac{d(y,x_0)^\gz}{R^\gz}
\frac{1}{{\mu(B(x_0,\sqrt t))^{1/2}\mu(B(y,\sqrt t))^{1/2}}}
e^{-\frac{d(y,x_0)^2}{C_1t}}\,d\mu(y)\nonumber\\
&&\hs+e^{-cR^2/t}\|u\|_{L^\fz(4B)}\int_{X\setminus
2B}\frac{1}{\mu(B(x_0,\sqrt t))^{1/2}\mu(B(y,\sqrt t))^{1/2}}
 e^{-\frac{d(y,x_0)^2}{2C_1t}}\,d\mu(y)\nonumber \\
&& \le
C\lf[ C(u,g)R^{-\gz}t^{\gz/2}+\|u\|_{L^\fz(4B)}e^{{-cR^2/t}}\r]\int_X p(lt,x_0,y)\,d\mu(y)\nonumber\\
&&\le
CC(u,g)\lf[R^{-\gz}t^{\gz/2}\r],
\end{eqnarray*}
where $l=\frac{2C_1}{C_2}$, which proves the second inequality and implies that
\begin{eqnarray*}
|u(x)\psi(x)-T_t(u\psi)(x_0)|&&\le |u(x)\psi(x)-u(x_0)\psi(x_0)|+|T_t(u(x_0)\psi(x_0))-T_t(u\psi)(x_0)|\nonumber\\
&&\le CC(u,g)\lf[R^{-\gz}(d(x,x_0)^\gz+t^{\gz/2})\r].
\end{eqnarray*}
The proof is then completed.
\end{proof}

\begin{prop}\label{p3.1}
There exists $C=C(Q,C_P(8R),C_Q(8R))>0$ such that for almost
every $x_0\in B$
\begin{equation*}
J(x_0,R^2)\le C\lf(\frac 1{R^2}\|u\|_{L^\fz(8B)}^2+R^2\|g(\cdot,u)\|_{L^\fz(8B)}^2\r).
\end{equation*}
\end{prop}

\begin{proof} For each $0<\ez<t\le R^2$, set
\begin{eqnarray}\label{3.2x}
J_\ez(x_0,t):=&&\frac{1}{t}\int_{\ez}^t\int_X |D(u\psi)(x)|^2
p(s,x_0,x)\,d\mu(x)\,ds.
\end{eqnarray}
Notice that
$$|D(u\psi)|^2=\frac 12 A (u\psi)^2-A(u\psi)T_s(u\psi)(x_0)-[u\psi-T_s(u\psi)(x_0)](\psi Au+uA\psi+2Du\cdot D\psi)$$
in the weak sense of measures. 
Thus, for each $0<\ez<t$ we have
\begin{eqnarray}\label{3.2}
tJ_\ez(x_0,t)&&=\frac 12 \int_\ez^t\int_X
[A((u\psi)^2)(x)-2A(u\psi)(x)T_s(u\psi)(x_0)]
p(s,x_0,x)\,d\mu(x)\,ds\nonumber\\
&&\hs -\int_\ez^t\int_X [(u\psi)(x)-T_s(u\psi)(x_0)][\psi
Au+uA\psi+2Du\cdot D\psi]p(s,x_0,x)\,d\mu(x)\,ds.
\end{eqnarray}
As the heat kernel is a solution to the heat equation $Au=\frac{\partial}{\partial t}u$ on $(\ez,t)\times X$ (see
Sturm \cite[proposition 2.3]{st2}), we further deduce that
\begin{eqnarray}\label{3.3x}
&&\int_\ez^t\int_X  [A((u\psi)^2)(x)-2A(u\psi)(x)T_s(u\psi)(x_0)]p(s,x_0,x)\,d\mu(x)\,ds\nonumber\\
&&\hs=[T_t((u\psi)^2)(x_0)-T_\ez((u\psi)^2)(x_0)]-\int_\ez^t 2T_s(u\psi)(x_0)AT_s(u\psi)(x_0)\,ds\nonumber\\
&&\hs=[T_t((u\psi)^2)(x_0)-T_\ez((u\psi)^2)(x_0)]-\int_\ez^t \frac{\partial}{\partial s}[ T_s(u\psi)(x_0)]^2\,ds\nonumber\\
&&\hs=\lf(T_t((u\psi)^2)(x_0)-[T_t(u\psi)(x_0)]^2\r)-\lf(T_\ez((u\psi)^2)(x_0)-[T_\ez(u\psi)(x_0)]^2\r).
 \end{eqnarray}

As the functions $(u\psi)(x)-T_s(u\psi)(x_0)$, $p(s,x_0,x)$ and
$\psi$ are bounded functions with gradient in $L^2(X)$, and
$\supp\lf\{ [(u\psi)(x)-T_s(u\psi)(x_0)]\psi p(s,x_0,x)\r\}\subset
4B$, we deduce from Lemma \ref{l3.2} that
\begin{eqnarray*}
&&\lf|\int_\ez^t\int_X [(u\psi)(x)-T_s(u\psi)(x_0)]\psi Au
p(s,x_0,x)\,d\mu(x)\,ds\r|\\
&&=\lf|\int_\ez^t\int_X [(u\psi)(x)-T_s(u\psi)(x_0)]\psi g(x,u(x))
p(s,x_0,x)\,d\mu(x)\,ds\r|\\
&&\le \lf|\int_\ez^t\int_X
[(u\psi)(x)-(u\psi)(x_0)-T_s(u\psi-(u\psi)(x_0))(x_0)]\psi g(x,u(x))
p(s,x_0,x)\,d\mu(x)\,ds\r|\\
 &&\le \|\psi g(\cdot,u)\|_{L^\fz(4B)}\int_\ez^t T_s(|u\psi-(u\psi)(x_0)|)(x_0)\,ds\\
&&\le C\|
g(\cdot,u)\|_{L^\fz(4B)}\lf[\|u\|_{L^\fz(4B)}+R^2\|g(\cdot,u)\|_{L^\fz(4B)}\r]\int_0^{t}R^{-\gz}s^{\gz/2}\,ds\\
&&\le C\|
g(\cdot,u)\|_{L^\fz(4B)}\lf[\|u\|_{L^\fz(4B)}+R^2\|g(\cdot,u)\|_{L^\fz(4B)}\r]R^{-\gz}t^{1+\gz/2}.
\end{eqnarray*}

The estimates for second  and third terms in \eqref{3.2} follow from
the following Lemma \ref{l3.4},
\begin{eqnarray}\label{3.3}
&&\lf|\int_\ez^t\int_X [(u\psi)(x)-T_s(u\psi)(x_0)]p(s,x_0,x)[u(x) A\psi(x)+2Du(x)\cdot
D\psi(x)]\,d\mu(x)\,ds\r|\nonumber\\
&& \hs\le C
e^{-cR^2/t}\|u\|_{L^\fz(8B)}\lf(\|u\|_{L^\fz(8B)}+R^2\|g(\cdot,u)\|_{L^\fz(8B)}\r).
\end{eqnarray}

As the underlying space is stochastically complete, the H\"older inequality implies for each $t>0$,
 $$T_t((u\psi)^2)(x_0)-[T_t(u\psi)(x_0)]^2\ge 0.$$
Combining the above estimates, by \eqref{3.2}, we obtain that
\begin{eqnarray*}
J_\ez(x_0,t)&&\le \frac 1{2t} \lf|\int_\ez^t\int_X
[A((u\psi)^2)(x)-2A(u\psi)(x)T_s(u\psi)(x_0)]
p(s,x_0,x)\,d\mu(x)\,ds\r|\nonumber\\
&&\hs\hs+\frac 1t\lf|\int_\ez^t\int_X [(u\psi)(x)-T_s(u\psi)(x_0)]
[\psi Au+uA\psi+2Du\cdot D\psi]p(s,x_0,x)\,d\mu(x)\,ds\r|\nonumber\\
&&\le \frac 1{2t}\lf\{T_t((u\psi)^2)(x_0)-[T_t(u\psi)(x_0)]^2\r\}+\frac 1{2t}\lf\{T_\ez((u\psi)^2)(x_0)-[T_\ez(u\psi)(x_0)]^2\r\}\nonumber\\
&&\hs+C\frac
{e^{-cR^2/t}}t\lf(\|u\|_{L^\fz(8B)}^2+R^2\|u\|_{L^\fz(8B)}\|g(\cdot,u)\|_{L^\fz(8B)}\r)\nonumber\\
&&\hs + C\|
g(\cdot,u)\|_{L^\fz(4B)}\lf(\|u\|_{L^\fz(4B)}+R^2\|g(\cdot,u)\|_{L^\fz(4B)}\r)R^{-\gz}t^{\gz/2}\nonumber\\
&&\le \frac 1{2t}\lf\{T_t((u\psi)^2)(x_0)-[T_t(u\psi)(x_0)]^2\r\}+\frac 1{2t}\lf\{T_\ez((u\psi)^2)(x_0)-[T_\ez(u\psi)(x_0)]^2\r\}\nonumber\\
&&\hs+ CR^{-\gz}t^{\gz/2}\lf[\frac 1{R^2}\|u\|_{L^\fz(8B)}^2+R^2\|g(\cdot,u)\|_{L^\fz(8B)}^2\r].
 \end{eqnarray*}
Finally, observe that for almost every $x_0$, it holds
$$\lim_{\ez\to 0}\lf\{T_\ez((u\psi)^2)(x_0)-[T_\ez(u\psi)(x_0)]^2\r\}=(u\psi)(x_0)^2-(u\psi)(x_0)^2=0;$$
see also the following inequality \eqref{3.7x}.
Hence, the monotone convergence theorem gives us
\begin{eqnarray}\label{3.4}
J(x_0,t)=\lim_{\ez\to 0}J_\ez(x_0,t)&&\le \frac 1{2t}\lf\{T_t((u\psi)^2)(x_0)-[T_t(u\psi)(x_0)]^2\r\}\nonumber\\
&&\hs+ CR^{-\gz}t^{\gz/2}\lf[\frac 1{R^2}\|u\|_{L^\fz(8B)}^2+R^2\|g(\cdot,u)\|_{L^\fz(8B)}^2\r].
\end{eqnarray}
Letting $t=R^2$ completes the proof of Proposition \ref{p3.1}.
\end{proof}

\begin{rem}\rm In \cite{krs}, it was proved that
\begin{eqnarray*}
&&\int_0^t\int_X  [A((u\psi)^2)(x)-2A(u\psi)(x)T_s(u\psi)(x_0)]p(s,x_0,x)\,d\mu(x)\,ds=
T_t((u\psi)^2)(x_0)-[T_t(u\psi)(x_0)]^2,
 \end{eqnarray*}
which was also used in \cite{ji2}. The proof of the equality
needs a careful argument to deal with the singularity of $-\frac{\partial}{\partial t}p+A_x p(\cdot,x_0,\cdot)$
at $(0,x_0,x_0)$; see \cite[proposition 3.4]{krs}. As pointed out by Kell \cite{ke13},
an upper bound for measure of balls with
small radius is needed in the proof, i.e.,
$\mu(B(x,r))\le Cr$
for $r<1$, see \cite[p.160]{krs}. We do not know if this is true in our settings.

Our proof of Proposition 3.1 above avoids using this equality, and is more direct and easier.
\end{rem}

To estimate the remaining term \eqref{3.3} in Proposition \ref{p3.1} we recall the Caccioppoli
inequality, which follows by inserting a suitable test function into the equation.

\begin{lem}[Caccioppoli inequality]\label{l3.3}
 If $L u=g(\cdot, u)$ in $B(y_0,R)$, then there exists $C>0$ independent of $Q,C_Q,C_P$
 such that for every $r<R$ it holds
\begin{equation*}
\int_{B(y_0,r)}|Du|^2\,d\mu\le
\frac{C}{(R-r)^2}\int_{B(y_0,R)}|u|^2\,d\mu+\int_{B(y_0,R)}|g(x,u(x))||u(x)|\,d\mu(x).
\end{equation*}
\end{lem}

\begin{lem}\label{l3.4}
There exists $C=C(Q,C_Q(8R),C_P(8R))>0$ such that for almost all
$x_0\in B=B(y_0,R)$, $x\in 2B$ and all $0<\ez<t\le R^2$, it holds
\begin{eqnarray*}
&&\lf|\int_\ez^t\int_X
[(u\psi)(x)-T_s(u\psi)(x_0)]p(s,x_0,x)[u(x) A\psi(x)+2Du(x)\cdot
D\psi(x)]\,d\mu(x)\,ds\r|\\
&&\hs\le Ce^{\frac{-R^2}{ct}}\|u\|_{L^\fz(8B)}
\lf(\|u\|_{L^\fz(8B)}+R^2\|g(\cdot,u)\|_{L^\fz(8B)}\r).
\end{eqnarray*}
\end{lem}
\begin{proof} Let $w(x,s)=u(x)\psi(x)-T_s(u\psi)(x_0)$ and notice that
\begin{eqnarray*}
&&w(x,s)p(s,x_0,x)[u(x) A\psi(x)+2Du(x)\cdot
D\psi(x)]\\
&&\hs=\lf[w(x,s)p(s,x_0,x)Du(x)-D(u\psi)(x) u(x)p(s,x_0,x)
-w(x,t)u(x)Dp(s,x_0,x)\r]\cdot D\psi(x)\\
&&\hs=-T_s(u\psi)(x_0)p(s,x_0,x)Du(x)\cdot D\psi(x)-|D\psi(x)|^2 u(x)^2p(s,x_0,x)
\\
&&\hs\hs-w(x,t)u(x)Dp(s,x_0,x)\cdot D\psi(x)\\
&&\hs=:\mathrm{H}_1+\mathrm{H}_2+\mathrm{H}_3
\end{eqnarray*}
in the weak sense of measures.
By using the Caccioppoli
inequality (Lemma \ref{l3.3}) and the H\"older inequality, we obtain
\begin{eqnarray*}
&&\lf|\int_\ez^t\int_X \mathrm{H}_1 \,d\mu(x)\,ds\r|\\
&&\le  \int_\ez^t\int_{B(y_0,{4R})\setminus
B(y_0,{2R})} |T_s(u\psi)(x_0)||Du(x)||D\psi(x)|p(s,x_0,x)\,d\mu(x)\,ds
\\&& \le \frac CR \|u\|_{L^\fz(4B)}\int_0^t\int_{B(x_0,{5R})\setminus B(x_0,{R})}
|Du(x)|\frac {R^Q}{s^{Q/2}\mu(B(x_0,R))}e^{\frac{-R^2}{cs}}\,d\mu(x)\,ds\\
&&\le \frac {Ct}R  \frac
{1}{\mu(B(x_0,R))}e^{\frac{-R^2}{ct}}\|u\|_{L^\fz(4B)}
\int_{B(x_0,{5R})\setminus B(x_0,{R})}
|Du(x)|\,d\mu(x)\\
&&\le \frac {Ct}R  \frac
{1}{\mu(B(x_0,R))^{1/2}}e^{\frac{-R^2}{ct}}\|u\|_{L^\fz(4B)}
\mu(B(x_0,4R))^{1/2}\lf\{\||u|\|^{1/2}_{L^\fz(8B)}\|g(\cdot,u)\|^{1/2}_{L^\fz(8B)}
+\frac {\|u\|_{L^\fz(8B)}}{R}\r\}\\
&&\le Ce^{\frac{-R^2}{ct}}
\lf(\|u\|_{L^\fz(8B)}^2+R^2\|u\|_{L^\fz(8B)}\|g(\cdot,u)\|_{L^\fz(8B)}\r).
\end{eqnarray*}

By \eqref{2.2}, we have
\begin{eqnarray*}
\lf|\int_\ez^t\int_X \mathrm{H}_2\,d\mu(x)\,ds\r|
&&\le \frac C{R^2} \|u\|^2_{L^\fz(4B)} \int_0^t\int_{B(y_0,4R)\setminus B(y_0,{2R})} \frac {1}{\mu(B(x_0,\sqrt s))}e^{-\frac{d(x,x_0)^2}{cs}}\,d\mu(x)\,ds
\\&& \le \frac C{R^2} \|u\|^2_{L^\fz(4B)} \int_0^t\int_{B(x_0,5R)\setminus B(x_0,{R})}
\frac {R^Q}{s^{Q/2}\mu(B(x_0,R))}e^{\frac{-R^2}{cs}}\,d\mu(x)\,ds\\
&&\le Ce^{\frac{-R^2}{ct}}\|u\|_{L^\fz(8B)}^2.
\end{eqnarray*}

We use Lemma \ref{l2.2} and the H\"older inequality  to estimate the
last term,
\begin{eqnarray*}
&&\lf|\int_\ez^t\int_X \mathrm{H}_3 \,d\mu(x)\,ds\r|\\
&&\le  \frac CR \int_0^t\int_{B(y_0,{4R})\setminus B(y_0,{2R})}
\lf|[u(x)\psi(x)-T_s(u\psi)(x_0)]u(x)\r| |Dp(s,x_0,x)|\,d\mu(x)\,ds\\
&&\le  \frac CR \|u\|_{L^\fz{(4B)}}^2
\int_0^t\int_{B(x_0,{5R})\setminus B(x_0,{R})}
|Dp(s,x_0,x)|\,d\mu(x)\,ds\\
&&\le   \frac CR \|u\|_{L^\fz{(4B)}}^2 C\sqrt t \mu(B(x_0,R))^{1/2}
\mu(B(x_0,R))^{-1/2}\exp\lf\{-\frac{R^2}{ct}\r\}\\
&&\le Ce^{\frac{-R^2}{ct}}\|u\|_{L^\fz{(8B)}}^2,
\end{eqnarray*}
which completes the proof.
\end{proof}

\begin{lem}\label{l3.5}
There exists $C=C(Q,C_P(4R),C_Q(4R))>0$ such that almost
all $x_0\in B=B(y_0,R)$ it holds
\begin{eqnarray*}
&&\int_0^{R^2} \lf\{T_t((u\psi)^2)(x_0)-[T_t(u\psi)(x_0)]^2\r\}\frac
{\,dt}t\le C\lf[\|u\|_{L^\fz(4B)}+R^2\|g(\cdot,u)\|_{L^\fz(4B)}\r]^2.
\end{eqnarray*}
\end{lem}

\begin{proof}
Let $C(u,g)=\|u\|_{L^\fz(4B)}+R^2\|g(\cdot,u)\|_{L^\fz(4B)}$. By Lemma \ref{l3.2}, we deduce that
\begin{eqnarray}\label{3.7x}
\quad&&T_t((u\psi)^2)(x_0)-[T_t(u\psi)(x_0)]^2\nonumber\\
&&\hs= \int_{2B}(u\psi(x)-T_t(u\psi)(x_0))^2p(t,x_0,x)\,d\mu(x)
+\int_{X\setminus 2B}(u\psi(x)-T_t(u\psi)(x_0))^2p(t,x_0,x)\,d\mu(x)\nonumber\\
&&\hs \le CC(u,g)^2
\int_{2B}\frac{(d(x,x_0)^\gz+t^{\gz/2})^2}
{R^{2\gz}\mu(B(x_0,\sqrt t))^{1/2}\mu(B(x,\sqrt t))^{1/2}}e^{-\frac{d(x,x_0)^2}{2C_1t}}
e^{-\frac{d(x,x_0)^2}{2C_1t}}\,d\mu(x)\nonumber\\
&&\hs\hs+C\|u\|_{L^\fz(4B)}^2\int_{X\setminus 2B} \frac{1}
{\mu(B(x_0,\sqrt t))^{1/2}\mu(B(x,\sqrt t))^{1/2}}e^{-\frac{d(x,x_0)^2}{2C_1t}}
e^{-\frac{d(x,x_0)^2}{2C_1t}}\,d\mu(x)\nonumber\\
&&\hs \le CC(u,g)^2
\frac{t^{\gz}}{R^{2\gz}}\int_{2B}p(lt,x_0,x)\,d\mu(x)+Ce^{-cR^2/t}\|u\|_{L^\fz(4B)}^2
\int_{X\setminus 2B}p(lt,x_0,x)\,d\mu(x)\nonumber\\
&&\hs\le C\lf[\|u\|_{L^\fz(4B)}+R^2\|g(\cdot,u)\|_{L^\fz(4B)}\r]^2
\frac{t^{\gz}}{R^{2\gz}},
\end{eqnarray}
where $l=\frac{2C_1}{C_2}\ge 2$ and we used the doubling condition
that $\frac{1}{\mu(B(x_0,\sqrt t))}\le C_Q(4R)l^Q\frac{1}{\mu(B(x_0,l\sqrt t))}$. From
this, we further conclude that
\begin{eqnarray*}
&&\int_0^{R^2} \lf[T_t((u\psi)^2)(x_0)-[T_t(u\psi)(x_0)]^2\r]\frac
{\,dt}t\le C\lf[\|u\|_{L^\fz(4B)}+R^2\|g(\cdot,u)\|_{L^\fz(4B)}\r]^2,
\end{eqnarray*}
which completes the proof.
\end{proof}

We remark here that Lemmas \ref{l3.1}-\ref{l3.5} and Proposition
\ref{p3.1} only require a doubling condition and an $L^2$-Poincar\'e inequality.

\begin{proof}[Proof of Theorem \ref{t3.1}]
 Let us first estimate the derivative $J'(x_0,t)=\frac{\,d}{\,dt}J(x_0,t)$.
By \eqref{3.4}, we deduce that
\begin{eqnarray*}
\frac{\,d}{\,dt}J(x_0,t)&&=-\frac{1}{t^2}J(x_0,t)+\frac{1}{t}
\int_X |D(u\psi)(x)|^2 p(t,x_0,x)\,d\mu(x)\nonumber\\
&&\hs\ge \frac{1}{t}\lf(\int_X |D(u\psi)(x)|^2 p(t,x_0,x)\,d\mu(x)-
\frac 1{2t}\lf\{T_t((u\psi)^2)(x_0)-[T_t(u\psi)(x_0)]^2\r\}\r)\nonumber\\
&&\hs\hs-\frac {Ct^{\gz/2-1}}{R^\gz}\lf[\frac 1{R^2}\|u\|_{L^\fz(8B)}^2+R^2\|g(\cdot,u)\|_{L^\fz(8B)}^2\r].
\end{eqnarray*}

For each fixed $t\in (0,R^2)$, either
$$\int_X |D(u\psi)(x)|^2 p(t,x_0,x)\,d\mu(x)\ge
\frac 1{2t}\lf\{T_t((u\psi)^2)(x_0)-[T_t(u\psi)(x_0)]^2\r\}$$ or
$$\int_X |D(u\psi)(x)|^2 p(t,x_0,x)\,d\mu(x)<
\frac 1{2t}\lf\{T_t((u\psi)^2)(x_0)-[T_t(u\psi)(x_0)]^2\r\}.$$ In
the first case, we have
\begin{eqnarray}\label{3.5}
\frac{\,d}{\,dt}J(x_0,t)\ge&& -\frac {Ct^{\gz/2-1}}{R^\gz}
\lf[\frac 1{R^2}\|u\|_{L^\fz(8B)}^2+R^2\|g(\cdot,u)\|_{L^\fz(8B)}^2\r].
\end{eqnarray}
In the second case, by the curvature condition \eqref{1.4} with
$T=R^2$, we deduce that
\begin{eqnarray}\label{3.6}
\frac{\,d}{\,dt}J(x_0,t)&&\ge-c_{\kz}(R^2)\int_X |D(u\psi)(x)|^2
p(t,x_0,x)\,d\mu(x) \nonumber\\
&&\hs\hs-\frac {Ct^{\gz/2-1}}{R^\gz}
\lf[\frac 1{R^2}\|u\|_{L^\fz(8B)}^2+R^2\|g(\cdot,u)\|_{L^\fz(8B)}^2\r]\nonumber\\
&&\ge-\frac{c_{\kz}(R^2)}{2t}\lf\{T_t((u\psi)^2)(x_0)-[T_t(u\psi)(x_0)]^2\r\}\nonumber\\
&&\hs\hs-\frac {Ct^{\gz/2-1}}{R^\gz}\lf[\frac 1{R^2}\|u\|_{L^\fz(8B)}^2+R^2\|g(\cdot,u)\|_{L^\fz(8B)}^2\r].
\end{eqnarray}
From \eqref{3.5} and \eqref{3.6}, we see that \eqref{3.6} holds in
both cases. Integrating over $(0,R^2)$ and applying Lemma
\ref{l3.5} we conclude that
\begin{eqnarray*}
\int_0^{R^2}J'(x_0,t)\,dt&&\ge-\int_0^{R^2}\frac{c_{\kz}(R^2)}{2t}
\lf\{T_t((u\psi)^2)(x_0)-[T_t(u\psi)(x_0)]^2\r\}\,dt\nonumber\\
&&\hs\hs + C\lf[\frac 1{R^2}\|u\|_{L^\fz(8B)}^2+R^2\|g(\cdot,u)\|_{L^\fz(8B)}^2\r]\\
&&\ge -C\lf(\frac
1{R^2}+c_{\kz}(R^2)\r)\lf[\|u\|_{L^\fz(8B)}^2+
R^4\|g(\cdot,u)\|^2_{L^\fz(8B)}\r].
\end{eqnarray*}
By Proposition \ref{p3.1} and the fact that $\lim_{t\to
0^+}J(x_0,t)=|Du(x_0)|^2$ $\mu$ a.e., we obtain that for almost
every $x_0\in B$,
\begin{eqnarray*}
|Du(x_0)|^2&&=J(x_0,R^2)-\int_0^{R^2} \frac{\,d}{\,dt}J(x_0,t)\,dt\\
&&\le C\lf(\frac 1{R^2}+c_{\kz}(R^2)\r)\lf[\|u\|_{L^\fz(8B)}^2+ R^4\|g(\cdot,u)\|_{L^\fz(8B)}^2\r],
\end{eqnarray*}
which completes the proof of Theorem \ref{t3.1}.
\end{proof}

\section{Proof of the main results}
\hskip\parindent
Recall that a Sobolev function $u\in H^{1,2}(\Omega)$ is called non-negative {\it sub-harmonic}, $Lu\ge 0$, in $\Omega$ if
$u\ge 0$ on $\Omega$, and
\begin{equation}
-\int_\Omega Du(x)\cdot D\phi(x)\,d\mu(x)\ge 0,\ \ \forall \phi\in
H_0^{1,2}(\Omega) \ \mathrm{and} \ {\phi\ge 0}.
\end{equation}
Similarly, $u$ is called non-negative {\it super-harmonic}, $Lu\le 0$, in $\Omega$ if
$u\ge 0$ on $\Omega$, and
\begin{equation}
-\int_\Omega Du(x)\cdot D\phi(x)\,d\mu(x)\le 0,\ \ \forall \phi\in
H_0^{1,2}(\Omega) \ \mathrm{and} \ {\phi\ge 0}.
\end{equation}

\begin{lem}\label{l4.1}
Let $(X,d)$ be a metric space with a locally $Q$-doubling measure
$\mu$, $Q\in (1,\fz)$. Assume that $(X,d,\mu)$ supports a local
weak $L^2$-Poincar\'e inequality. Let $u$ be a solution
to the equation $Lu=-\lz u$ in $2B$, where $B=B(y_0,2R)$ and $\lz\in L^\fz(2B)$.
Then for each $p>0$, there exists $C=C(Q,C_Q(2R),C_P(2R),p,\|\lz\|_{L^\fz(2B)} R^2)$ such that
$$\|u\|_{L^\fz{(B(y_0,R))}}\le C\lf(\fint_{B(y_0,2R)}|u|^p\,d\mu\r)^{1/p}.$$
\end{lem}
\begin{proof}
From Lemma \ref{l3.3}, we have the following Caccioppoli inequality
\begin{eqnarray*}
\int_{B(y_0,r_1)}|Du|^2\,d\mu&&\le
\frac{C}{(r_2-r_1)^2}\int_{B(y_0,r_2)}|u|^2\,d\mu+\int_{B(y_0,r_2)}|(\lz u)(x)||u(x)|\,d\mu(x)\\
&&\le \frac{C(1+\|\lz\|_{L^\fz(2B)}R^2)}{(r_2-r_1)^2}\int_{B(y_0,r_2)}|u|^2\,d\mu,
\end{eqnarray*}
for arbitrary $0<r_1<r_2\le 2R$.
By using the Caccioppoli inequality and the Sobolev inequality (see \cite{hak,sal,st3}),
the  proof follows via the Moser iterations technique and \cite[Lemma 5.2]{bm};
see \cite{lmp06} for instance. We omit the details here.
\end{proof}
\begin{rem}\rm
For $Q\in (1,2)$, we can have a better estimate by combining the Caccioppoli inequality
and the Sobolev-Poincar\'e inequality from \cite{hak}. Precisely,
\begin{eqnarray*}
\|u\|_{L^\fz(B)}&&\le \|u\phi\|_{L^\fz(B(y_0,3R/2))}\le \fint_{B(y_0,3R/2)} |u|\,d\mu
+CR\lf(\fint_{B(y_0,3R/2)}|D(u\phi)|^2\,d\mu\r)^{1/2}\\
&&\le C(1+\sqrt{\|\lz\|_{L^\fz(2B)}}R)\lf(\fint_{B(y_0,2R)}|u|^2\,d\mu\r),
\end{eqnarray*}
where $\phi$ is a cut-off function that equals one on $B(y_0,R)$ and is supported  in $B(y_0,3R/2)$.

For $Q\in (2,\fz)$, the Moser iteration would give us that
$$\|u\|_{L^\fz{(B(y_0,R))}}\le C(Q,C_P(2R),C_Q(2R))
\lf(1+\|\lz\|_{L^\fz(2B)}R^2\r)^{Q/4}\lf(\fint_{B(y_0,2R)}|u|^2\,d\mu\r)^{1/2}.$$
The above estimate implies that for the case $Q=2$, it holds
$$\|u\|_{L^\fz{(B(y_0,R))}}\le C(Q,C_P(2R),C_Q(2R))
\lf(1+\|\lz\|_{L^\fz(2B)}R^2\r)^{\ez+1/2}\lf(\fint_{B(y_0,2R)}|u|^2\,d\mu\r)^{1/2},$$
for arbitrary $\ez>0$.

We choose to avoid the precise dependence on $\lz$ mainly because our main aim in the paper
is to give gradient estimates for harmonic functions, and also to avoid complicated calculations.
\end{rem}

\begin{lem}\label{l4.2}
Let $(X,d)$ be a metric space with a locally $Q$-doubling measure
$\mu$, $Q\in (1,\fz)$. Assume that $(X,d,\mu)$ supports a local
weak $L^2$-Poincar\'e inequality. Let $u$ be a positive super-harmonic
function in $B(y_0,2R)$.
Then there exists $q>0$ and $C=C(Q,C_Q(2R),C_P(2R),q)$ such that
$$\lf(\fint_{B(y_0,2R)}u^q\,d\mu\r)^{1/q}\le \inf_{x\in B(y_0,R) }u(x).$$
\end{lem}
\begin{proof}
From \cite[Proposition 5.7]{bm}, we see that there exists $q>0$ such that
$$\fint_{B(y_0,2R)}u^q\,d\mu\fint_{B(y_0,2R)}u^{-q}\,d\mu\le C.$$
Also from the Moser iteration, it holds
$$\lf(\fint_{B(y_0,2R)}u^{-q}\,d\mu\r)^{-1/q}\le C\inf_{x\in B(y_0,R) }u(x);$$
see \cite[pp.162-163]{bm} for instance. The lemma follows from the above
two inequalities.
\end{proof}


\begin{proof}[Proof of Theorem \ref{t1.1}]
For each $x\in B(y_0,R)$, we choose a ball $B(x,{R/16})$.
Then  for each $z\in B(x,R)$, it holds
$$d(z,y_0)\le d(z,x)+d(x,y_0)<2R,$$
and hence $B(x,R)\subset B(y_0,2R)$. Lemma \ref{l4.1} and the doubling condition imply that
$$\|u\|_{L^\fz{(B(x,R/2))}}\le C\fint_{B(x,R)}|u|\,d\mu
\le C\fint_{B(y_0,2R)}|u|\,d\mu.$$
By Theorem \ref{t3.1}, we obtain that
\begin{eqnarray*}
\||Du|\|_{L^\fz(B(x,R/16))}
&&\le C\lf(\frac 1{R}+\sqrt {c_{\kz}(R^2)}\r)\lf[\|u\|_{L^\fz(B(x,R/2))}+ R^2 \| \lz u\|_{L^\fz(B(x,R/2))}\r]\\
&& \le C(Q,C_P(2R),C_Q(2R),\|\lz\|_{L^\fz(2B)}R^2)\lf(\frac 1{R}+\sqrt{c_{\kz}(R^2)}\r)\fint_{B(y_0,2R)}|u|\,d\mu.
\end{eqnarray*}
Thus $$\||Du|\|_{L^\fz{(B(y_0,R))}}\le C(Q,C_P(2R),C_Q(2R),\|\lz\|_{L^\fz(2B)}R^2)\lf(\frac 1{R}+\sqrt
{c_{\kz}(R^2)}\r)\fint_{B(y_0,2R)}|u|\,d\mu,$$
which proves Theorem \ref{t1.1}.
\end{proof}


\begin{proof}[Proof of Theorem \ref{t1.2}]
Let $u$ be a positive harmonic function $u$ on $B(y_0,{2R})$.
By Theorem \ref{t1.1} and Lemma \ref{l4.2}, we have
\begin{equation}\label{4.3}
\|Du\|_{L^\fz(B(y_0,R))}\le
C(Q,C_Q(2R),C_P(2R))\lf(\sqrt {c_\kz(R^2)}+\frac{1}{R}\r)\inf_{x\in B(y_0,R)}u(x),
\end{equation}
which implies that for almost every $x\in B(y_0,R)$,
\begin{equation*}
|Du(x)|\le C(Q,C_Q(2R),C_P(2R))\lf(\sqrt{c_\kz(R^2)}+\frac{1}{R}\r)u(x),
\end{equation*}
which proves (i).

Let $c_\kz(1)>0$. If $R<1/2$, then in the estimate we can always use
$C_Q(1),C_P(1),c_\kz(1)$ to replace $C_Q(2R),C_P(2R),c_\kz(2R)$ and obtain
\begin{equation*}
|Du(x)|\le C(Q,C_Q(1),C_P(1))\lf(\sqrt{c_\kz(1)}+\frac{1}{R}\r)u(x).
\end{equation*}

If $R>1/2$,  then for each point $x\in B(y_0,{R})$, we choose a ball
$B(x,1/8)$. Then $B(x,1/2)\subset B(y_0,{2R})$, by \eqref{4.3}, we
have
\begin{equation*}
\|Du\|_{L^\fz(B(x,1/8))}\le
C(Q,C_Q(1),C_P(1))\lf(\sqrt{c_\kz(1)}+1\r)\inf_{y\in B(x,1/8)}u(y),
\end{equation*}
which implies that for almost every $x\in B(y_0,{R})$, it holds
\begin{equation}\label{4.6}
|Du(x)|\le C(Q,C_Q(1),C_P(1))\lf(\sqrt{c_\kz(1)}+\frac 1R\r)u(x),
\end{equation}
which completes the proof.
\end{proof}

\subsection*{Acknowledgment}
\hskip\parindent  The author would like to thank his advisor Professor
Pekka Koskela for helpful comments on the paper.
R. Jiang was supported by the Academy of Finland Grant 131477.

\noindent Renjin Jiang

\noindent
School of Mathematical Sciences, Beijing Normal University,
Laboratory of Mathematics and Complex Systems, Beijing 100875, People's Republic of China

and

\noindent Department of Mathematics and Statistics, University of Jyv\"{a}skyl\"{a}, P.O. Box 35 (MaD),
FI-40014
Finland

\noindent{\it E-mail address}: \texttt{rejiang@bnu.edu.cn}


\begin{thebibliography}{999}

\vspace{-0.3cm}
\bibitem{agmr}   L. Ambrosio, N. Gigli, A. Mondino, T. Rajala,
Riemannian Ricci curvature lower bounds in metric measure spaces with $\sz$-finite measure,
 arXiv:1207.4924.

\vspace{-0.3cm}
\bibitem{ags1} L. Ambrosio, N. Gigli, G. Savar\'e, Metric measure spaces
with Riemannian Ricci curvature bounded from below,
arXiv:1109.0222.


\vspace{-0.3cm}
\bibitem{bak1} D. Bakry, On Sobolev and Logarithmic Inequalities for Markov Semigroups, New Trends in
Stochastic Analysis (Charingworth, 1994), World Scientific
Publishing, River Edge, NJ, 1997, pp. 43-75.

\vspace{-0.3cm}
\bibitem{be2} D. Bakry, M. Emery, Diffusions hypercontractives, Seminaire de probabilities,
Vol. XIX, 1983/84, pp. 177-206.



\vspace{-0.3cm}
\bibitem{bm} M. Biroli, U. Mosco, A Saint-Venant type principle for
Dirichlet forms on discontinuous media, Ann. Mat. Pura Appl. 169
(1995) 125-181.

\vspace{-0.3cm}
\bibitem{bb10}
A. Bj\"orn, J. Bj\"orn, Nonlinear Potential Theory on Metric Spaces,
to appear in EMS Tracts in Mathematics, European Mathematical
Society, Zurich.


\vspace{-0.3cm}
\bibitem{ck} L.A. Caffarelli, C.E. Kenig, Gradient estimates for
variable coefficient parabolic equations and singular perturbation
problems, Amer. J. Math. 120 (1998) 391-439.

\vspace{-0.3cm}
\bibitem{ch} J. Cheeger, Differentiability of Lipschitz functions on
metric measure spaces, Geom. Funct. Anal. 9 (1999), 428-517.

\vspace{-0.3cm}
\bibitem{chy} S.Y. Cheng, S.T. Yau, Differential equations on Riemannian
manifolds and their geometric applications, Comm. Pure Appl. Math.
28 (3) (1975) 333-354.


\vspace{-0.3cm}
\bibitem{fot} M. Fukushima, Y. Oshima, M. Takeda, Dirichlet Forms and Symmetric
Markov Processes, in: de Gruyter Studies in Mathematics, Vol. 19,
Walter de Gruyter \& Co., Berlin, 1994.


\vspace{-0.3cm}
\bibitem{gko11} N. Gigli, K. Kuwada, S.I. Ohta, Heat flow on Alexandrov spaces,
Comm. Pure Appl. Math.,  66 (2013), 307-331.

\vspace{-0.3cm}
\bibitem{gro1} L. Gross, Logarithmic Sobolev inequalities, Amer. J. Math. 97 (1975) 1061-1083.


\vspace{-0.3cm}
\bibitem{hak} P. Haj{\l}asz, P. Koskela, Sobolev met Poincar\'e,
Mem. Amer. Math. Soc. 145 (2000).


\vspace{-0.3cm}
\bibitem{hek} J. Heinonen, P. Koskela, Quasiconformal maps in metric spaces
with controlled geometry,  Acta Math.  181  (1998), 1-61.

\vspace{-0.3cm}
\bibitem{li08}
P. Li, Harmonic functions on complete Riemannian manifolds. Handbook
of geometric analysis, No. 1,  195¨C227, Adv. Lect. Math. (ALM), 7,
Int. Press, Somerville, MA, 2008.

\vspace{-0.3cm}
\bibitem{ji2} R. Jiang, Gradient estimates of solutions to Poisson equation
in metric measure spaces,J. Funct. Anal. 261 (2011), 3549-3584.



\vspace{-0.3cm}
\bibitem{kei03} S. Keith, Modulus and the Poincar\'e inequality
on metric measure spaces,  Math. Z.  245 (2003), 255-292.


\vspace{-0.3cm}
\bibitem{ke13}
M. Kell, A Note on Lipschitz Continuity of Solutions of Poisson Equations in Metric Measure Spaces, arXiv:1307.2224.

\vspace{-0.3cm}
\bibitem{kkm} T. Kilpel\"ainen, J. Kinnunen, O. Martio,
Sobolev spaces with zero boundary values on metric spaces,
Potential Anal. 12 (3) (2000) 233-247.

 \vspace{-0.3cm}
\bibitem{krs} P. Koskela, K. Rajala, N. Shanmugalingam, Lipschitz
continuity of Cheeger-harmonic functions in metric measure spaces,
J. Funct. Anal. 202 (2003), 147-173.

\vspace{-0.3cm}
\bibitem{kms} K. Kuwae, Y. Machigashira, T. Shioya, Sobolev spaces,
Laplacian, and heat kernel on Alexandrov spaces,  Math. Z. 238
(2001), 269-316.

\vspace{-0.3cm}
\bibitem{lmp06}  V. Latvala, N. Marola, M. Pere, Harnack's
inequality for a nonlinear eigenvalue problem on metric spaces,
J. Math. Anal. Appl. 321 (2006), 793-810.

\vspace{-0.3cm}
\bibitem{ly86} P. Li, S.T. Yau, On the parabolic kernel of the
Schr\"odinger operator, Acta Math. 156 (1986), 153-201.



\vspace{-0.3cm}
\bibitem{lv09} J. Lott, C. Villani, Ricci curvature for metric-measure spaces
via optimal transport,  Ann. of Math. (2)  169  (2009), 903-991.

%
\vspace{-0.3cm}
\bibitem{raj1} T. Rajala,
Local Poincar\'e inequalities from stable curvature conditions on
metric spaces, Calc. Var. Partial Differential Equations 44 (2012),
477-494.

\vspace{-0.3cm}
\bibitem{raj2} T. Rajala,
Interpolated measures with bounded density in metric spaces
satisfying the
 curvature-dimension conditions of Sturm, J. Funct. Anal. 263 (2012), 896?24.

\vspace{-0.3cm}
\bibitem{sh} N. Shanmugalingam, Newtonian spaces: an extension of Sobolev
spaces to metric measure spaces,  Rev. Mat. Iberoamericana 16
(2000), 243-279.

\vspace{-0.3cm}
\bibitem{sal} L. Saloff-Coste, A note on Poincar\'e, Sobolev,
and Harnack inequalities, Internat. Math. Res. Notices (2) (1992)
27-38.


\vspace{-0.3cm}
\bibitem{st1} K.T. Sturm, Analysis on local Dirichlet spaces. I. Recurrence,
conservativeness and $L^p$-Liouville properties, J. Reine Angew.
Math. 456 (1994) 173-196.

\vspace{-0.3cm}
\bibitem{st2} K.T. Sturm, Analysis on local Dirichlet spaces. II. Upper
Gaussian estimates for the fundamental solutions of parabolic
equations, Osaka J. Math. 32 (2) (1995) 275-312.

\vspace{-0.3cm}
\bibitem{st3} K.T. Sturm, Analysis on local Dirichlet spaces. III. The
parabolic Harnack inequality, J. Math. Pures Appl. (9) 75 (3) (1996)
273-297.

\vspace{-0.3cm}
\bibitem{stm4} K.T. Sturm, On the geometry of metric measure spaces I,
Acta Math. 196 (2006), 65-131.

\vspace{-0.3cm}
\bibitem{stm5}
K.T. Sturm, On the geometry of metric measure spaces II, Acta Math.
196 (2006), 133-177.

\vspace{-0.3cm}
\bibitem{zz10} H.C. Zhang, X.-P. Zhu, On a new definition of Ricci curvature on
Alexandrov spaces, Acta Math. Sci. Ser. B Engl. Ed. 30 (2010),
1949-1974.

\vspace{-0.3cm}
\bibitem{zz11} H.C. Zhang, X.-P. Zhu, Yau's gradient estimates on Alexandrov spaces,
J. Differential Geometry, 91 (2012), 445-522.

\vspace{-0.3cm}
\bibitem{ya75} S.T. Yau, Harmonic functions on complete Riemannian manifolds,
Comm. Pure Appl. Math. 28 (1975), 201-228.
\end{thebibliography}
\end{document}